\documentclass[11pt, a4paper]{amsart}
\usepackage{amsmath}
\usepackage{amssymb}
\usepackage{color}

\newtheorem{teo}{Theorem}[section]

\theoremstyle{definition}

\def\ep{\varepsilon}

\def\a{\alpha}
\def\R{\mathbb R}

\begin{document}
\title[Asymptotics for a fully nonlocal heat equation]{Large-time behavior for a fully nonlocal heat equation}

\author[Cort\'{a}zar,  Quir\'{o}s \and Wolanski]{Carmen Cort\'{a}zar,  Fernando  Quir\'{o}s, \and Noemi Wolanski}

\address{Carmen Cort\'{a}zar\hfill\break\indent
	Departamento  de Matem\'{a}tica, Pontificia Universidad Cat\'{o}lica
	de Chile \hfill\break\indent Santiago, Chile.} \email{{\tt
		ccortaza@mat.puc.cl} }

\address{Fernando Quir\'{o}s\hfill\break\indent
	Departamento  de Matem\'{a}ticas, Universidad Aut\'{o}noma de Madrid,
	\hfill\break\indent 28049-Madrid, Spain,
	\hfill\break\indent and Instituto de Ciencias Matem\'aticas ICMAT (CSIC-UAM-UCM-UC3M),
	\hfill\break\indent 28049-Madrid, Spain.} \email{{\tt
		fernando.quiros@uam.es} }

\address{Noemi Wolanski \hfill\break\indent
	IMAS-UBA-CONICET, \hfill\break\indent Ciudad Universitaria, Pab. I,\hfill\break\indent
	(1428) Buenos Aires, Argentina.} \email{{\tt wolanski@dm.uba.ar} }

\thanks{This project has received funding from the European Union's Horizon 2020 research and innovation programme under the Marie Sklodowska-Curie grant agreement No.\,777822. C.\,Cort\'azar supported by  FONDECYT grant 1190102 (Chile). F.\,Quir\'os supported by the Spanish Ministry of Science and Innovation, through  project MTM2017-87596-P and through the  \lq\lq Severo Ochoa Programme for Centres of Excellence in R\&D'' (SEV-2015-0554), and by the Spanish National Research Council, through the \lq\lq Ayuda extraordinaria a Centros de Excelencia Severo Ochoa'' (20205CEX001). N.\,Wolanski supported by
	CONICET PIP625, Res. 960/12, ANPCyT PICT-2012-0153, UBACYT X117 and MathAmSud 13MATH03 (Argentina).}

\keywords{Fully nonlocal heat equation, Caputo derivative, fractional Laplacian, asymptotic behavior.}

\subjclass[2010]{%
	35B40, 
	35R11, 
	35R09, 
	45K05. 
}

\date{}

\begin{abstract}
We study the large-time behavior in  all $L^p$ norms and in different space-time scales of solutions to  a nonlocal heat equation in $\mathbb{R}^N$ involving a Caputo  $\alpha$-time derivative and a power of the Laplacian $(-\Delta)^s$, $s\in (0,1)$, extending recent results by the authors for the case $s=1$. The initial data are assumed to be integrable, and, when required, to be also in $L^p$. The main novelty with respect to the case $s=1$ comes from the behaviour in fast scales, for which, thanks to the fat tails of the fundamental solution of the equation, we are able to give results that are not available neither for the case $s=1$ nor, to our knowledge, for the standard heat equation, $s=1$, $\alpha=1$.
\end{abstract}

\maketitle


\section{Introduction}
\label{sect-Introduction} \setcounter{equation}{0}
The aim of this paper is to study the large-time behavior of solutions to the \emph{fully nonlocal} Cauchy problem
\begin{equation}
\label{eq:main}
\partial_t^\a u+(-\Delta)^s u=0\quad\mbox{in }\R^N\times(0,\infty),\qquad
u(\cdot,0)=u_0\ge0\quad\mbox{in }\R^N,
\end{equation}
where $\partial_t^\alpha$, $\alpha\in(0,1)$, denotes the Caputo $\alpha$-derivative, introduced in~\cite{Caputo-1967}, defined for smooth functions by 
$$
\displaystyle\partial_t^\alpha u(x,t)=\frac1{\Gamma(1-\alpha)}\,\partial_t\int_0^t\frac{u(x,\tau)- u(x,0)}{(t-\tau)^{\alpha}}\, d\tau, 
$$
and $(-\Delta)^s$, $s\in(0,1)$,  is the usual $s$ power of the Laplacian, defined for smooth functions by
$$
\displaystyle (-\Delta)^s v(x)=c_{N,s}\operatorname{P.V.}\int_{\R^N}\frac{v(x)-v(y)}{|x-y|^{N+2s}}\,dy.
$$
Here $c_{N,s}$ is a positive normalization constant, chosen so that $(-\Delta)^{s}=\mathcal{F}^{-1}(|\cdot|^{2s}\mathcal{F})$, where $\mathcal{F}$ denotes Fourier transform.
The initial data are always assumed to be non-negative, integrable and non-trivial, and when required, to  belong also to  $L^p(\mathbb{R}^N)$ for some $p\in (1,\infty]$.

Fully nonlocal heat equations, like~\eqref{eq:main},  nonlocal both in space and time, are useful to model situations with long-range interactions and memory effects, and have been proposed for example to describe plasma transport~see~\cite{delCastilloNegrete-Carreras-Lynch-2004, delCastilloNegrete-Carreras-Lynch-2005}; see also~\cite{Compte-Caceres-1998,Metzler-Klafter-2000,Zaslavsky-2002,Cartea-delCastilloNegrete-2007} for further models that use such equations. 

Despite their wide number of
 applications,  the mathematical study of fully nonlocal heat equations started only quite recently.  In the interesting paper~\cite{Kemppainen-Siljander-Zacher-2017}, Kemppainen, Siljander and Zacher proved, among several other things,  that if both $u_0$ and $\mathcal{F}(u_0)$ belong to $L^1(\mathbb{R}^N)$, then  problem~\eqref{eq:main} has a unique bounded classical solution given by
\begin{equation}
\label{eq:convolution}
 u(x,t)=\int_{\R^N}Z(x-y,t)u_0(y)\,dy,
\end{equation}
where $Z$ is the fundamental solution of the equation, which  has a Dirac mass as initial datum. 
If $u_0$ is only known to be in $L^1(\mathbb{R}^N)$, problem~\eqref{eq:main} does not have in general a classical solution. However, the function $u$ in~\eqref{eq:convolution} is still well defined, and is a solution of~\eqref{eq:main} in a generalized sense~\cite{Gripenberg-1985,Kemppainen-Siljander-Zacher-2017}. It is to this kind of solutions,  sometimes denoted in the literature as \emph{mild} solutions, that we will always refer to in the present paper. See also~\cite{Gorenflo-Luchko-Yamamoto-2015} for a notion of weak solution. 

\medskip

\noindent\emph{Notations. } As is common in asymptotic analysis, $f\asymp g$ will mean that there are constants $\nu,\mu>0$ such that
$\nu g\le f\le \mu g$. The spatial integral of the solution, which is preserved along the evolution, and is hence equal to $\int_{\mathbb{R}^N}u_0$, will be denoted by $M $. The ball of radius $r$ centered at the origin will be denoted by $B_r$.

\medskip

The main difficulty when studying the large-time behavior of solutions to~\eqref{eq:main}, as compared with the case $\alpha=1$, is that $Z$ is singular in space at the origin if $2s\le N$, with a singularity as that of the fundamental solution for $(-\Delta)^s$ in dimension $N$, 
\begin{equation*}
\label{eq:fundamental.solution}
E_{N,s}(\xi)=\begin{cases}
|\xi|^{2s-N},&2s<N,\\
-\ln |\xi|,&2s=N=1.
\end{cases}
\end{equation*}
To be more precise, $Z$ has a self-similar form,
$$
Z(x,t)=t^{-\frac{\alpha N}{2s}}F(\xi),\quad \xi=x t^{-\frac{\alpha}{2s}},
$$
for a radially symmetric positive profile $F\in C^\infty(\mathbb{R}^N\setminus\{0\})$ satisfying
\begin{eqnarray}
\label{eq:global.bounds.F}
&F(\xi)\le C
\begin{cases}
E_{N,s}(\xi),&2s<N,\\
1+|E_{1,\frac12}(\xi)|,&2s=N=1,\\
1,&2s>N=1,
\end{cases}\qquad &\xi\neq0,\\[8pt]
\label{eq:bounds.F.infty}
&F(\xi)\le C |\xi|^{-(N+2s)},\quad N\ge 1, s\in(0,1),\qquad &|\xi|\ge 1,
\end{eqnarray}
and also
\begin{eqnarray}
\label{eq:constant.Nge2s}
F(\xi)/E_{N,s}(\xi)\to\kappa\quad&\text{as }|\xi|\to0,\quad& 2s\le N,\\[8pt]
\notag
F(\xi)\to\kappa\quad&\text{as }|\xi|\to0,\quad&2s> N=1,\\[8pt]
\label{eq:constant.infty}
|\xi|^{N+2s}	F(\xi)\to \hat\kappa\quad&\text{as }|\xi|\to\infty,\quad &N\ge 1, s\in(0,1),
\end{eqnarray}
for some positive constants $\kappa=\kappa(N,s)$ and $\hat\kappa=\hat\kappa(N,s)$; see~\cite{Kim-Lim-2016,Kemppainen-Siljander-Zacher-2017}. In particular $F$ is continuous at the origin if and only if $2s>N=1$, and  $Z(\cdot,t)\in L^p(\mathbb{R}^N)$  for $t>0$ if and only if $p$ is \emph{subcritical},
\begin{equation}
\label{eq:subcritical.range}
\tag{S}
p\in [1,\infty]\quad\text{if }2s>N=1,\qquad p\in[1,p_c)\quad\text{if }N\ge 2s, \quad \text{where }
p_c=\begin{cases}
\frac{N}{N-2s}&\text{if }N>2s,\\
\infty&\text{if }N=2s.
\end{cases}
\end{equation}
The papers~\cite{Kim-Lim-2016,Kemppainen-Siljander-Zacher-2017} also provide estimates for the gradient,
\begin{equation}
\label{eq:estimates.gradient}
\begin{array}{ll}
|DF(\xi)|\le C
\begin{cases}
|\xi|^{2s-1-N},&2s\le N,\\
1,&2s>N=1,
\end{cases}\qquad &\xi\neq0,\\[8pt]
|DF(\xi)|\le C|\xi|^{-(N+2s+1)},\qquad &|\xi|\ge 1,
\end{array}
\end{equation}
and in fact for all subsequent derivatives.

If $p$ is subcritical, that is, if~\eqref{eq:subcritical.range} holds, there is an $L^1$--$L^p$ smoothing effect: if the initial datum is in $L^1(\mathbb{R}^N)$, then $u(\cdot,t)\in L^1(\mathbb{R}^N)\cap L^p(\mathbb{R}^N)$ for every $t>0$. Moreover, we have the (sharp) decay rate~$\|u(\cdot,t)\|_{L^p(\mathbb{R}^N)}\asymp t^{-\frac{\alpha N}{2s}(1-\frac1p)}$; see~\cite{Kemppainen-Siljander-Zacher-2017}. In that same paper the authors also proved that in that range
$$
t^{\frac{\alpha N}{2s}(1-\frac1p)}\|u(\cdot,t)-M Z(\cdot,t)\|_{L^p(\R^N)}\to0\quad\mbox{as }t\to\infty,
$$
under the additional technical assumption $p<N/(N-2s+1)$. We will get rid of this restriction in Theorem~\ref{L^p}.

In the critical case $N>2s$,  $p=p_c$,  $Z(\cdot,t)$ belongs to the Marcinkiewicz space
$M^{p_c}(\mathbb{R}^N)$, which coincides with the Lorentz space $L^{p,\infty}(\mathbb{R}^N)$, defined through the norm
\[
\|g\|_{M^p(\mathbb{R}^N)}=\sup\Big\{\lambda\big|\{x\in\R^n:|g(x)|>\lambda\}\big|^{1/p},\ \lambda>0\Big\}, \quad p\in [1,\infty),
\]
and there is a weak smoothing effect: solutions with an integrable initial datum will belong to~$L^1(\mathbb{R}^N)\cap M^{p_c}(\mathbb{R}^N)$ for any later time. As proved also in \cite{Kemppainen-Siljander-Zacher-2017},  we have the decay rate~$\|u(\cdot,t)\|_{M^{p_c}(\mathbb{R}^N)}\asymp t^{-\alpha}$, but no global asymptotic profile in $M^{p_c}$-norm was known.

In the supercritical range, $2s<N$, $p>p_c$,  there is no smoothing effect, and in order to have $u(\cdot, t)\in L^p(\mathbb{R}^N)$ we have to require in addition $u_0\in L^p(\mathbb{R}^N)$. Under this extra assumption, it was proved in~\cite{Kemppainen-Siljander-Zacher-2017,Cheng-Li-Yamamoto-2017} by means of Fourier techniques that
$$
\|u(\cdot,t)\|_{L^2(\mathbb{R}^N)}\asymp t^{-\alpha},
$$ 
a result which was shown to be sharp.
In this supercritical range there is no information available for $p\neq2$, nor about asymptotic profiles. However, the above results are enough to  observe the interesting \emph{critical dimension phenomenon}: the decay rate does not improve when the dimension is increased after $N>4s$. This phenomenon, which was already observed when $s=1$ in~\cite{Kemppainen-Siljander-Vergara-Zacher-2016}, is due to the memory introduced by the $\alpha$-Caputo derivate, and does not take place when $\alpha=1$. 

In order to obtain further information on the large-time behavior of solutions to~\eqref{eq:main} we will follow a different approach, used very recently by the authors in~\cite{CQW} to deal with the case $s=1$, namely, to study the decay rates and profiles separately for different space-time scales.

The equation has a \emph{characteristic scale}, $|x|\asymp t^{\frac\alpha{2s}}$, dictated by its scaling invariance.
The behavior in this scale is given by $MZ$, no matter the value of $p$,
$$
t^{\frac{\alpha N}{2s}(1-\frac1p)}\|u(\cdot,t)-M
Z(\cdot,t)\|_{L^p(\{|x|\asymp t^{\frac\a{2s}}\})}\to0\quad\mbox{as }t\to\infty,
$$ 
under the additional decay assumption
$u_0\in \mathcal{D}_N$ if $p$ is not subcritical, where
$$
\mathcal{D}_\beta=\{f: \text{there exist } C,R>0 \text{ such that } |x|^\beta|f(x)|\le C\text{ for all } x\in \mathbb{R}^N, |x|>R\}.
$$
This asymptotic result implies the sharp decay rate
$$
\|u(\cdot,t)\|_{L^p(\{|x|\asymp t^{\frac\a{2s}}\})}\asymp t^{-\frac{\alpha N}{2s}(1-\frac1p)}.
$$

In \emph{intermediate scales}, $|x|\asymp g(t)$ with $g$ \emph{growing slowly to infinity},
\begin{equation}
\label{eq:intermediate.scales}
g(t)\to\infty,\quad g(t)=o(t^{\frac\alpha{2s}})\quad\text{as }t\to\infty,
\end{equation}
the large-time behavior is still given by $M$ times the fundamental solution, 
$$
\frac{\|u(\cdot,t)-M
	Z(\cdot,t)\|_{L^p(\{|x|\asymp g(t)\})}}{\|u(\cdot,t)\|_{L^p(\{|x|\asymp g(t)\})}}\to0\quad\mbox{as }t\to\infty,
$$
again with the additional assumption $u_0\in\mathcal{D}_N$ if $p$ is not subcritical,
but now the decay rate depends on the precise scale and on the size of $2s$ relative to $N$, with logarithmic corrections in the critical case $2s=N=1$. 
$$
\|u(\cdot,t)\|_{L^p(\{|x|\asymp g(t)\})}\asymp\begin{cases}
\displaystyle
t^{-\alpha} (g(t))^{2s-N\big(1-\frac 1p\big)},&2s<N,\\[10pt]
\displaystyle\frac{t^{-\a}(g(t))^{
		\frac1p}}{|\log (g(t)t^{-\alpha})|},&2s=N=1,\\[10pt]
\displaystyle t^{-\frac\alpha{2s}}(g(t))^{\frac 1p},&2s>N=1.	
\end{cases}
$$
In these intermediate scales $\xi=xt^{-\frac\a{2s}}\to0$, and hence one may use the behavior of the profile $F$ at the origin to give a more precise description of the asymptotic limit. Thus, solutions will approach a multiple of $E_{N,s}$ if $2s<N$, and a constant if $2s\ge N=1$; see Theorem~\ref{teo-intermediate-supercritical}.

It is in compact sets, and for $2s<N$, where we find the main qualitative difference with respect to the standard local heat equation: the decay rate is $t^{-\alpha}$, independently of the values of $s$ and $N$ (and $p$), and the final profile is a multiple of the Newtonian potential of the initial datum,
$$
\Phi(x):=\int_{\mathbb{R}^N}\frac{u_0(x-y)}{|y|^{N-{2s}}}\,dy,
$$
assuming also $u_0\in L^1(\R^N)\cap L^{p}_{\rm loc}(\R^N)$ if $p$ is not subcritical, so that $\Phi\in L^p_{\rm loc}(\mathbb{R}^N)$; see Theorem~\ref{teo-compacto Lp subcritico} for a detailed statement. This phenomenon, which  had already been observed for the case $s=1$ in~\cite{CQW}, is due to the memory introduced by the Caputo derivative, and does not take place for $\alpha=1$.

For the critical value  $2s=N=1$ the  decay rate in compact sets is not given exactly by $t^{-\alpha}$; it has a logarithmic correction that makes the decay a bit slower. Moreover, the asymptotic profile is not given by the Newtonian potential of the initial datum, but by the constant $M\kappa\alpha$, with $\kappa$ as in~\eqref{eq:constant.Nge2s}. If $2s> N=1$, the asymptotics are also described by a constant, namely $MF(0)$,  with a decay rate~$t^{-\frac\alpha{2s}}$, where $F$ is the profile of the fundamental solution. The details for these two cases are given respectively in Theorems~\ref{teo-compactos N=2 todo p} and~\ref{thm:compact.N=1}.

The results that we have described so far, which are gathered in Section~\ref{sect-RN}, are definitely interesting. However, they are qualitatively similar to the ones obtained in~\cite{CQW} for the case $s=1$ and, once we have the required estimates for $Z$ at hand, their proofs can be adapted from the ones for that case. Hence, we  omit most of them,  and write down only those where some more significant difference appears.

The main novelty of this paper, comes when we consider \emph{fast scales}
$$
|x|\asymp g(t)\quad\text{with }
g(t)t^{-\frac\alpha{2s}}\to \infty\quad\text{as }t\to\infty.
$$
Our asymptotic results in these scales, given in Section~\ref{sect:fast.scales}, depend strongly on the decay of the initial datum at infinity, and are better than the ones available for $s=1$,  even when $\alpha=1$. The difference stems from the powerlike decay of the profile $F$ at infinity~\eqref{eq:constant.infty}, that contrasts with its exponential decay when $s=1$.

If $u_0$ decays at least as fast as the fundamental solution, $u_0\in L^1(\mathbb{R}^N)\cap\mathcal{D}_\beta$ with $\beta\ge N+2s$, the solution $u$ converges towards $MZ$  uniformly in relative error  in regions of the form 
$$
\{|x|\ge \nu t^{\frac\a{2s}}\},\quad\nu>0.
$$ 
When $s=1$, the best available results do not go so far:  even for  compactly supported initial data, which obviously decay faster than the fundamental solution,    the aforementioned convergence is only known to hold in sets of the form
$$ 
\{\nu t^{\a/2}\le |x|\le\mu t(\log t)^{\frac{2-\alpha}\alpha}\},\qquad \nu>0,\quad \mu\in\Big(0,\Big(\frac{\a(2-\alpha)}{4R\sigma}\Big)^{\frac{2-\alpha}\alpha}\Big), 
$$
where $R$ is the radius of a ball containing the support of $u_0$; see~\cite{CQW}.

If $u_0\in L^1(\mathbb{R}^N)\cap\mathcal{D}_\beta$, $\beta\in (N,N+2s)$, which means that it may decay more slowly than the fundamental solution, but yet in an integrable way,  uniform convergence in relative error towards $MZ$ can only be guaranteed in sets of the form
$$
\nu t^{\frac\a{2s}}\le |x|\le h(t),\qquad \nu>0,\quad h(t)=o\big(t^{\frac\a{N+2s-\beta}}).
$$
The upper restriction is not technical.  Indeed, if $|x|^\beta u_0(x)\to A>0$ as $|x|\to\infty$ for some  $\beta\in(N,N+2s)$, then $u$ converges towards $u_0$ as $t\to\infty$ uniformly in relative error in sets of the form $\{|x|\ge h(t)\}$ if  $h(t)/t^{\frac\a{N+2s-\beta}}\to\infty$ as $t\to\infty$. Let us remark that $|x|^{\beta}$ and $Z(x,t)$ are of the same order precisely if $|x|\asymp t^{\frac\a{N+2s-\beta}}$.

Let us note that  solutions decay in fast scales faster than in the characteristic scale, no matter the value of $p$ or the initial datum. Moreover, the decay rate in intermediate scales is intermediate between the one in compact sets and the one in the characteristic scale. It is easy to check that for any  given $p\in [1,\infty)$ the slowest decay rate is the one in the characteristic scale, $t^{-\frac{\alpha N}{2s}(1-\frac1p)}$, for dimensions $N<2sp/(p-1)$, and the one in compact sets, $t^{-\alpha}$,  if $N>2sp/(p-1)$. Therefore,  the critical dimension phenomenon holds for any value of $p\in[1,\infty)$: the decay rate increases with the dimension as long as $N<2sp/(p-1)$, after which it stays equal to $t^{-\alpha}$. For $p=\infty$ the threshold value is $2s$.

Let us finally point out that some nonlinear versions of~\eqref{eq:main} have been recently considered in~\cite{Allen-Caffarelli-Vasseur-2016,Allen-Caffarelli-Vasseur-2017}. An analysis of the large-time behavior for such models is a challenging problem that would require completely new tools.

\section{Slow scales}\label{sect-RN}
\setcounter{equation}{0}

In this section we consider the large-time behavior in characteristic and intermediate scales, and also in compact sets. Since the proofs are quite similar to the corresponding  ones for the case $s=1$,  we omit most of them, indicating its counterpart in~\cite{CQW}, and write down only those where some difference appears. 

\subsection{Characteristic scale}
We consider first the large-time behavior in $L^p(\mathbb{R}^N)$ when $p$ is in the subcritical range. This gives in particular the behavior in the characteristic scale.
\begin{teo}\label{L^p} Let $u_0\in L^1(\R^N)$. If $p$ is subcritical, then
	\begin{equation*}\label{lp}
	t^{\frac{\alpha N}{2s}(1-\frac1p)}\big\|u(\cdot,t)-M Z(\cdot,t)\big\|_{L^p(\R^N)}\to0\quad\mbox{as }t\to\infty.
	\end{equation*}
\end{teo}

\begin{proof} The proof follows the lines of that of  \cite[Theorem 2.1]{CQW}, since $F$ belongs to $L^p$ when $p$ is subcritical, and using also that $F'\in L^\infty(\mathbb{R})$ when dealing with the case $p=\infty$, $2s>N=1$.
\end{proof}
We remark that though the result is valid in the whole $\mathbb{R}^N$, it is only sharp in the characteristic scale. In other scales solutions decay faster for $p$ in this range.

If $p$ is not subcritical, the asymptotic behavior is still given by $MZ$ if we restrict to outer sets and assume some decay for the initial datum. This result is sharp in the characteristic scale.
	
\begin{teo}\label{teo-supercritical exterior} Let $N\ge 2s$ and $p\in [p_c,\infty]$. Assume $u_0\in L^1(\R^N)\cap \mathcal{D}_N$. Then, for any $\nu>0$,
	\[
	t^{\frac{\alpha N}{2s}(1-\frac1p)}\|u(\cdot,t)-M  Z(\cdot,t)\|_{L^p(\{|x|>\nu t^{\frac{\alpha}{2s}}\})}\to 0\quad\mbox{as }t\to\infty.
	\]
\end{teo}
\begin{proof} We follow the ideas in the proof of~\cite[Theorem 3.1]{CQW}, though we will have to make some changes, due to the differences in the fundamental solutions.  The starting point is the identity
\begin{equation}
\label{eq:starting.point}
t^{\frac{\alpha N}{2s}}\big(u(x,t)-M  Z(x,t)\big)
=\int u_0(y)\big(F((x-y)t^{-\frac{\a}{2s}})-F(xt^{-\frac{\a}{2s}})\big)\,dy.
\end{equation}
Then, for $N>2s$ and $p\in [p_c,\infty)$ we get
\[\begin{aligned}
t^{\frac{\alpha N}{2s}(1-\frac1p)}\big\|u(\cdot,t)&-M  Z(\cdot,t)\big\|_{L^p(\{|x|>\nu t^{\frac\alpha{2s}}\})}
\le {\rm I}+{\rm II}+{\rm III},\quad\text{where }\\
{\rm I}&=t^{-\frac{\alpha N}{2sp}}\int u_0(y)\Big(\int_{\scriptsize\!\begin{array}{c}|x|> \nu t^{\frac\alpha{2s}}\\ |y|<\delta|x|\end{array}}\big|F((x-y)t^{-\frac\a{2s}})-F(xt^{-\frac\a{2s}})\big|^p\,dx\Big)^{1/p}\,dy,\\
{\rm II}&=t^{-\frac{\alpha N}{2sp}}
\left(\int_{|x|> \nu t^{\frac\alpha{2s}}}\Big(\int_{|y|>\delta|x|} u_0(y)F(xt^{-\frac\a{2s}})\,dy\Big)^p\,dx\right)^{1/p},\\
{\rm III}&=t^{-\frac{\alpha N}{2sp}} \left(\int_{|x|> \nu t^{\frac\alpha{2s}}}\Big(\int_{|y|>\delta|x|} u_0(y)F((x-y)t^{-\frac\a{2s}})\,dy\Big)^p\,dx\right)^{1/p}.
\end{aligned}
\]
Performing the change of variables $\xi=x t^{-\frac\alpha{2s}}$ we get
$$
{\rm I}=
\int u_0(y)\Big(\int_{\scriptsize\!\!\!\begin{array}{c}|\xi|>\nu\\ |y|t^{-\frac\alpha{2s}}<\delta|\xi|\end{array}}\big|F(\xi-
yt^{-\frac\alpha{2s}})-F(\xi)\big|^p\,d\xi\Big)^{1/p}\,dy.
$$
Because of the Mean Value Theorem, 
$\big|F(\xi-
yt^{-\frac\alpha{2s}})-F(\xi)\big|=\big|DF(\xi-\theta yt^{-\frac\a{2s}})\cdot yt^{-\frac\a{2s}}\big|$ for some $\theta\in(0,1)$. Let $\delta\in(0,1/2)$. If  $|\xi|\ge \nu $ and $|y|t^{-\frac\alpha{2s}}\le\delta|\xi|$, then  $|\xi-\theta yt^{-\frac\a{2s}}|\ge |\xi|/2\ge\nu/2$ for all $\theta\in(0,1)$. Therefore, using the estimates for $|DF|$ in \eqref{eq:estimates.gradient}, 
$$
\big|F(\xi-
yt^{-\frac\alpha{2s}})-F(\xi)\big|\le \frac{C_\nu|y|t^{-\frac\alpha{2s}}}{|\xi-\theta yt^{-\a/2}|^{N+2s+1}}\le  \frac{C_\nu 2^{N+2s+1}\delta}{|\xi|^{N+2s}}.
$$
Therefore, 
\[
{\rm I}\le C M\delta\Big(\int_{|\xi|>\nu}\frac{d\xi}{|\xi|^{N+2s}}\Big)^{1/p}\le CM \delta.
\]

On the other hand,
\[
{\rm II}\le  \left(\int_{|\xi|>\nu}|F(\xi)|^p\,d\xi\right)^{1/p}\int_{|y|>\delta \nu t^{\frac\a{2s}}}u_0(y)\,dy =C\int_{|y|>\delta\nu t^{\frac\a{2s}}}u_0(y)\,dy.
\]

As for ${\rm III}$, we take $\ep\in(0,1)$ and $\gamma> 0$ to be chosen later and make also the change of variables $\eta=yt^{-\frac\alpha{2s}}$. Since $u_0\in\mathcal{D}_N$, using the estimates~\eqref{eq:global.bounds.F}--\eqref{eq:bounds.F.infty} for the profile $F$ we get
\[\begin{aligned}
{\rm III}
=
&\left(\int_{|\xi|> \nu}\Big(t^{\frac{\alpha N}{2s}}\int_{|\eta|>\delta|\xi|} u_0(\eta t^{\frac\a{2s}})F(\xi-\eta)\,d\eta\Big)^p\,d\xi\right)^{1/p}
\le {\rm III}_{1}+{\rm III}_{2},\quad\text{where }\\
&{\rm III}_{1}= C\left(\int_{|\xi|\ge \nu}\Big(\int_{\scriptsize\!\!\begin{array}{c}|\eta|>\delta|\xi|\\
	|\xi-\eta|<\gamma|\xi|^\ep\end{array}}|\xi-\eta|^{2s-N}|\eta|^{-N}\,d\eta\Big)^p\,d\xi\right)^{1/p},\\
&{\rm III}_{2}= C\left(\int_{|\xi|\ge \nu}\Big(t^{\frac{\alpha N}{2s}}\int_{\scriptsize\!\!\begin{array}{c}|\eta|>\delta|\xi|\\
	|\xi-\eta|>\gamma|\xi|^\ep\end{array}}u_0(\eta t^{\frac\a{2s}})|\xi-\eta|^{-(N+2s)}\,d\eta\Big)^p\,d\xi\right)^{1/p}.
\end{aligned}\]
Since $p\ge \frac N{N-2s}$, if we choose  $\gamma=\delta^{\frac{N+1}{2s}}$ we have
\[\begin{aligned}
{\rm III}_{1}&\le C \delta^{-N}\left(\int_{|\xi|\ge \nu}|
\xi|^{-Np}\Big(\int_{|\xi-\eta|<\gamma|\xi|^\ep}|\xi-\eta|^{2s-N}\,d\eta\Big)^p\,d\xi\right)^{1/p}\\
&\le C \delta^{-N}\gamma^{2s}\left(\int_{|\xi|\ge \nu}|\xi|^{-p(N-2s\ep)}\,d\xi\right)^{1/p}\le C \delta.
\end{aligned}\]

On the other hand, choosing $\ep<1$ but close enough to 1 so that $(N+2s)\ep p>N$, we get
\[\begin{aligned}
{\rm III}_{2}&\le C\gamma^{-(N+2s)}\left(\int_{|\xi|\ge \nu}|\xi|^{-(N+2s)p\varepsilon}\Big(t^{\frac{N\a}{2s}}\int_{|\eta|>\delta|\xi|}u_0(\eta
t^{\frac\a{2s}})\,d\eta\Big)^p\,d\xi\right)^{1/p}\\
&\le C\gamma^{-(N+2s)}\Big(\int_{|\xi|\ge \nu}|\xi|^{-(N+2s)p\varepsilon}\,d\xi\Big)^{1/p}\int_{|y|>\delta\nu t^{\frac\a{2s}}}u_0(y)\,dy\to 0\quad \text{as }t\to\infty,
\end{aligned}
\]
which completes the proof for $N>2s$ and $p$ finite.

The cases $N>2s$, $p=\infty$ follow by letting $p\to\infty$ in the estimates for finite $p$.

In order to deal with the case $N=2s=1$, $p=\infty$, starting from~\eqref{eq:starting.point} we get
\[
\begin{aligned}
t^{\alpha}\big|u(x,t)&-M  Z(x,t)\big|\le {\rm I}+{\rm II}+{\rm III},\quad\text{where }\\
{\rm I}&=
\int_{|y|<\delta t^{\a}}u_0(y)|F((x-y)t^{-\a})-F(xt^{-\a})|\,dy,
\\
{\rm II}&=\int_{|y|>\delta t^{\a}}u_0(y)F(xt^{-\a})\,dy,
\\
{\rm III}&=\int_{|y|>\delta t^{\a}}u_0(y)F((x-y)t^{-\a})\,dy.
\end{aligned}
\]

The Mean Value Theorem implies that there exists $\theta\in(0,1)$ such that
\[
|F((x-y)t^{-\a})-F(xt^{-\a})|=|F'(xt^{-\alpha}-\theta yt^{-\alpha})| \,|y|t^{-\alpha}.
\]
If  $|x|\ge \nu t^{\alpha}$ and $|y|\le\delta t^{\alpha}$, with $\delta<\nu/2$ to be chosen, then $|xt^{-\alpha}-\theta yt^{-\alpha}|\ge \nu/2$ for all  $\theta\in(0,1)$. 
Therefore, since  $F'$ is bounded outside $B_{\nu/2}$, 
$$
\textrm{I}\le M\delta\|F'\|_{L^\infty(\mathbb{R}\setminus B_{\nu/2})}<M\varepsilon
$$
if $\delta<\min\{\nu/2, \ep/\|F'\|_{L^\infty(\mathbb{R}\setminus B_{\nu/2})}\}$.

We now bound $\textrm{II}$. As $|x|t^{-\alpha}\ge \nu$ and $F$ is bounded outside $B_\nu$,
\[
\textrm{II}\le \displaystyle
\displaystyle \|F\|_{L^\infty(\mathbb{R}\setminus B_\nu)}
\int_{|y|>\delta t^{\a}}u_0(y)\,dy <\ep
\quad\text{for }t\ge t_0(\ep,\delta,\nu).
\]

Now we turn to $\textrm{III}$. We have $\textrm{III}=\textrm{III}_{1}+\textrm{III}_{2}$, where
\[
\begin{aligned}
\textrm{III}_{1}&=\int_{\scriptsize\!\!\!\!\begin{array}{c}|y|>\delta t^{\alpha}\\|x-y|<\gamma t^{\alpha}\end{array}}|y|u_0(y)\,\frac{F\big((x-y)t^{-\alpha}\big)}{|y|}\,dy,\\
\textrm{III}_{2}&=\int_{\scriptsize\!\!\!\!\begin{array}{c}|y|>\delta t^{\alpha}\\|x-y|>\gamma t^{\alpha}\end{array}}u_0(y)F\big((x-y)t^{-\alpha}\big)\,dy.
\end{aligned}
\]
Since $u_0\in\mathcal{D}_1$, using  the bound $F(\xi)\le -C\log|\xi|$ for $|\xi|\le 1/2$, we get
\[
\textrm{III}_{1}\le -  C\delta^{-1}t^{-\alpha}\int_0^{\gamma t^{\alpha}}\log (rt^{-\alpha})\,dr= -  C\delta^{-1}\int_0^{\gamma}\log \rho\,d\rho \le  C\delta^{-1}\gamma|\log \gamma|
\]
if $\gamma\le 1/2$. Hence, choosing $\gamma=\delta^{2}$ with $\delta\in(0,1/\sqrt2)$ small enough, then
$\textrm{III}_{1}\le C\delta|\log\delta|<\varepsilon$.

On the other  hand,
\[
\textrm{III}_{2}\le \|F\|_{L^\infty(\mathbb{R}\setminus B_{\delta^2})}\int_{|y|>\delta t^{\a}}u_0(y)\,dy<\ep\quad\mbox{if } t\ge t_0(\ep,\delta).
\]
\end{proof}

\subsection{Intermediate scales}
In  intermediate scales, $|x|\asymp g(t)$ with $g$ satifying~\eqref{eq:intermediate.scales}, 
solutions still approach $MZ$, as in the characteristic scale, though with a different rate, that depends on the size of $2s$ when compared to $N$. Since in these scales $\xi=xt^{-\frac\a{2s}}\to0$ as $t\to\infty$, the results can be rewritten  using the behavior of the profile $F$ at the origin, which also depends on the size of $2s$. Thus, solutions will approach a multiple of $E_{N,s}$ if $2s<N$, and a constant if $2s\ge N=1$, which one depending on the precise scale when we have equality. 

\begin{teo}\label{teo-intermediate-supercritical} Let $\kappa>0$ as in~\eqref{eq:constant.Nge2s}. Let $u_0\in L^1(\mathbb{R}^N)$ if $p$ is subcritical, or $u_0\in L^1(\mathbb{R}^N)\cap \mathcal{D}_N$ if $p\in[p_c,\infty]$. For any $g$ satisfying~\eqref{eq:intermediate.scales} and $0<\nu<\mu<\infty$:
	
\noindent {\rm (a)} If $2s<N$, then     	
\begin{equation*}
\label{eq-intermedio-p-1}
\begin{array}{ll}
\displaystyle
t^{\alpha} (g(t))^{N\big(1-\frac 1p\big)-2s}\|u(\cdot,t)-M
Z(\cdot,t)\|_{L^p(\{\nu<|x|/g(t)<\mu\})}\to0\quad&\mbox{as }t\to\infty,
\\[8pt]
\displaystyle(g(t))^{N\big(1-\frac 1p\big)-2s}\|t^{\alpha} u(\cdot,t)-M\kappa E_{N,s}
	\|_{L^p(\{\nu <|x|/g(t)<\mu\})}\to0\quad&\mbox{as } t\to\infty.	
\end{array}
\end{equation*}

\noindent{\rm (b)} If $2s=N=1$, then
\begin{equation*}
\label{eq-intermedio1 N=2}
\begin{array}{ll}
\displaystyle \frac{t^{\a}(g(t))^{-
		\frac1p}}{|\log (g(t)t^{-\alpha})|}
\|{u(\cdot,t)-M  Z(\cdot,t)}\|_{L^p(\{\nu <|x|/g(t)<\mu\})}\to0\quad&\mbox{as }t\to\infty,\\[8pt]
\displaystyle (g(t))^{-\frac 1p}\Big\|\frac{t^{\a}}{|\log(g(t)t^{-\alpha})|}u(\cdot,t)-M  \kappa\Big\|_{L^p(\{\nu <|x|/g(t)<\mu\})}\to0\quad&\mbox{as } t\to\infty.
\end{array}
\end{equation*}

\noindent{\rm (c)} If $2s>N=1$, then
$$
\begin{array}{ll}
\displaystyle t^{\frac\alpha{2s}}(g(t))^{-\frac 1p}\|u(\cdot,t)-M Z(\cdot,t)\|_{L^p(\R^N)}\to0\quad&\mbox{as }t\to\infty,\\[8pt]
\displaystyle(g(t))^{-\frac 1p}\|t^{\frac\alpha{2s}} u(\cdot,t)-MF(0)
\|_{L^p(\{\nu <|x|/g(t)<\mu\})}\to0\quad&\mbox{as } t\to\infty.
\end{array}
$$
\end{teo}
\begin{proof} 
The proofs of (a) and (b) are similar respectively to those of~\cite[Theorem 4.1]{CQW} and~\cite[Theorem 4.2]{CQW}. Case (c)~follows immediately from Theorem~\ref{L^p}.
\end{proof}

\subsection{Compact sets} As in the case $s=1$, $N>2$,  when $2s<N$ we have a striking result on compact sets: the decay rate is $t^{-\alpha}$, independently of the values of $s$ and $N$, and the final profile is a multiple of $\Phi$, the Newtonian potential of the initial datum. 
\begin{teo}\label{teo-compacto Lp subcritico}
Let $2s<N$ and  $\kappa>0$ as in~\eqref{eq:constant.Nge2s}. 

\noindent {\rm (a)} Let $p\in[1,\infty]$, $p\neq p_c$. Let $u_0\in L^1(\R^N)$ if $p$ is subcritical or $L^1(\R^N)\cap L^{p}_{\rm loc}(\R^N)$ otherwise. Then $\Phi\in L^p_{\rm loc}(\mathbb{R}^N)$ and
$\displaystyle\|t^\a u(\cdot,t)-\kappa\Phi\|_{L^p(B_\mu)}\to0$ as $t\to\infty$.

\noindent{\rm (b)} Let $g$ satisfy~\eqref{eq:intermediate.scales}. If  $u_0\in L^1(\R^N)$, then 
$\Phi\in M^{p_c}(\R^N)$ and 
$\|t^\a u(\cdot,t)-\kappa\Phi\|_{M^{p_c}(B_{g(t)})}\to0$ as $t\to\infty$.
If moreover  $u_0\in L^1(\R^N)\cap L^{p_c}_{\rm loc}(\R^N)$, then 	$\Phi\in L^{p_c}_{\rm loc}(\mathbb{R}^N)$ and 	
$\|t^\a u(\cdot,t)-\kappa\Phi\|_{L^{p_c}(B_{g(t)})}\to0$ as $t\to \infty$.
\end{teo}
\begin{proof} 
The proof follows those of \cite[Theorems 5.1--5.5]{CQW}.
\end{proof}

For the critical value  $2s=N=1$ the  decay rate has a logarithmic correction that makes the decay a bit slower than $t^{-\alpha}$. Moreover, the asymptotic profile is not given by the Newtonian potential of the initial datum, but by a constant. The same asymptotic constant gives the profile in expanding balls $B_{g(t)}$ as long as the decay rate is the same as in compact sets, which is the case as long as $\log g(t)/\log t\to0$.
\begin{teo}\label{teo-compactos N=2 todo p}  $2s=N=1$,  $\kappa>0$ as in~\eqref{eq:constant.Nge2s}, and  $u_0\in L^1(\R)$.
	\begin{itemize}
		\item[(a)] Assume in addition that $u_0\in L^q_{\rm loc}(\R)$ for some $q\in (1,\infty]$ if $p=\infty$. Then, for every $p\in[1,\infty]$ and  $\mu>0$,
		\[
		\Big\|\frac{t^\a}{\log t}u(\cdot,t)-M \kappa\a\Big\|_{L^p(B_\mu)}\to0\quad\mbox{as } t\to\infty.
		\]
		\item[(b)] Assume in addition that $u_0\in L^q(\R)$ for some some $q\in (1,\infty]$ if $p=\infty$. Then, for every $p\in[1,\infty]$ and $g:\mathbb{R}_+\to\mathbb{R}_+$ such that $g(t)\to\infty$ and  $\log g(t)/\log t\to0$ as $t\to\infty$,
		\[
		(g(t))^{-\frac1p}\Big\|\frac{t^\a}{\log t}u(\cdot,t)-M \kappa\a\Big\|_{L^p(B_{g(t)})}\to0\quad\mbox{as } t\to\infty.
		\]
	\end{itemize}
	\end{teo}

	\begin{proof} The proof follows that of \cite[Theorem 5.6]{CQW}.\end{proof}

	The result for $2s>1$ is a direct  corollary of Theorem~\ref{L^p}. The asymptotic profile is a constant, which coincides with the one giving the asymptotic behavior in intermediate scales. In fact, the result is valid in expanding balls $B_{g(t)}$, if $g$ grows slowly to infinity.
	\begin{teo}
		\label{thm:compact.N=1}
		Let $2s>N=1$, $u_0\in L^1(\mathbb{R})$, and $p\in [1,\infty]$.
		\begin{itemize}
			\item[(a)] For any  $\mu>0$, $
			\|t^{\frac\alpha{2s}} u(\cdot,t)-MF(0)
			\|_{L^p(B_\mu)}\to0$ as $t\to\infty$.
			\item[(b)] For any $g$ satisfying \eqref{eq:intermediate.scales}, $
			(g(t))^{-\frac1p}\|t^{\frac\alpha{2s}} u(\cdot,t)-MF(0)
			\|_{L^p(B_{g(t)})}\to0$ as $t\to\infty$.
		\end{itemize}
	\end{teo}

\section{Fast scales}\label{sect:fast.scales}
\setcounter{equation}{0}

It is in fast scales that we get the more novel results of the paper, differing qualitatively from those available for the case $s=1$. The difference stems from the large tail of the fundamental solution.

Our first result analyzes the large-time behavior  for solutions with initial data that decay at infinity at least as fast as the fundamental solution. 

\begin{teo}
\label{thm:fast.decay} Let $u_0\in L^1(\R^N)\cap \mathcal{D}_\beta$, with $\beta\ge N+2s$. Then for every $\nu>0$,	
\begin{equation*}\label{eq-beta-large}
u(x,t)=MZ(x,t)\big(1+o(1)\big)\quad\mbox{uniformly in }|x|\ge \nu t^{\frac\a{2s}}\quad\text{as }t\to\infty.
\end{equation*}
\end{teo}
\begin{proof} We start by observing that, since  $F$ is positive, its behavior at infinity~\eqref{eq:constant.infty} implies that for all $\nu>0$ there is a constant $C_\nu>0$ such that 
\begin{equation}
\label{eq:estimate.F.below}
F(\xi)\ge C_\nu |\xi|^{-(N+2s)}\quad\text{for all }|\xi|\ge \nu. 
\end{equation}

Next, we split the error as
	$$
	\begin{aligned}
	|u(x,t)-M& Z(z,t)|\le  {\rm I}(x,t)+{\rm II}(x,t)+{\rm III}(x,t),\quad \text{where}\\
	{\rm I}(x,t)&=t^{-\frac{\a N}{2s}}\int_{|y|\le \delta(t)|x|} \big|F\big((x-y)t^{-\frac\a{2s}})-F(xt^{-\frac\a{2s}}\big)\big|u_0(y)\,dy,\\
	{\rm II}(x,t)&=t^{-\frac{\a N}{2s}}\int_{|y|\ge \delta(t)|x|} F\big((x-y)t^{-\frac\a{2s}}\big)u_0(y)\,dy,\\
	{\rm III}(x,t)&=t^{-\frac{\a N}{2s}}\int_{|y|\ge \delta(t)|x|} F\big(xt^{-\frac\a{2s}}\big)u_0(y)\,dy,
	\end{aligned}
	$$
	for some function $\delta=\delta(t)$ that will be specified later.

If  $|x|\ge \nu t^{\frac\a{2s}}$ and $|y|\le\delta(t)|x|$,  with
 $\delta(t)\in(0,1/2)$ for all $t\in\mathbb{R}_+$, then  
$$
 |xt^{-\frac\a{2s}}-\theta yt^{-\frac\a{2s}}|\ge |x|t^{-\frac\a{2s}}/2\ge\nu/2\quad\text{for all }\theta\in(0,1).
 $$
Therefore, using the estimates for $|DF|$ in \eqref{eq:estimates.gradient} and~\eqref{eq:estimate.F.below},
$$
\begin{aligned}
\big|F\big((x-y)t^{-\frac\a{2s}})-F(xt^{-\frac\a{2s}}\big)\big|&\le \frac{C_\nu|y|t^{-\frac\alpha{2s}}}{|xt^{-\frac\a{2s}}-\theta yt^{-\frac\a{2s}}|^{N+2s+1}}\le C 	\delta(t)|xt^{-\frac\a{2s}}|^{-(N+2s)}\\
&\le  C\delta(t)F(xt^{-\frac\a{2s}}).
\end{aligned}
$$
Hence, if $\delta(t)=o(1)$ as $t\to\infty$, we conclude that
$$
{\rm I}(x,t)\le C\delta(t) M t^{-\frac{\a N}{2s}}F(xt^{-\frac\a{2s}})=o(1)MZ(x,t)\quad\text{as }t\to\infty.
$$

On the other hand, using the decay assumption on the initial datum,
\begin{equation}
\label{eq:estimate.II}
\begin{aligned}
{\rm II}(x,t)&\le C t^{-\frac{\a N}{2s}}\int_{|y|>\delta(t)|x|}\frac{F((x-y)t^{-\frac\a{2s}})}{|y|^\beta}\,dy\le \frac{C t^{-\frac{\a N}{2s}}}{(\delta(t)|x|)^\beta}\int F((x-y)t^{-\frac\a{2s}})\,dy\\
&= C(\delta(t))^{-\beta}|x|^{-\beta}= C(\delta(t))^{-\beta}t^{-\frac{\a\beta}{2s}}|xt^{-\frac\a{2s}}|^{-(N+2s)}|xt^{-\frac\a{2s}}|^{-(\beta-N-2s)}.
\end{aligned}
\end{equation}
Hence, if $\beta\ge N+2s$ and  $|x|\ge \nu t^{\frac\a{2s}}$, $\nu>0$,  using~\eqref{eq:estimate.F.below} we get
$$
\begin{aligned}
{\rm II}(x,t)&\le C\nu^{-(\beta-N-2s)}(\delta(t))^{-\beta}t^{-\frac{\a\beta}{2s}}|xt^{-\frac\a{2s}}|^{-(N+2s)}
\\
&\le C(\delta(t))^{-\beta}t^{-\frac\a{2s}(\beta-N)}MZ(x,t)=o(1)MZ(x,t)
\end{aligned}
$$
if $\delta(t) t^{\frac\a{2s}\big(1-\frac N\beta\big)}\to\infty$.

Finally, if $|x|>\nu t^{\frac\a{2s}}$ with $\nu>0$, using once more~\eqref{eq:estimate.F.below}, 
$$
{\rm III}(x,t)\le Ct^{-\frac{\a N}{2s}} F\big(xt^{-\frac\a{2s}}\big)\int_{|y|>\delta(t)\nu t^{\frac\a{2s}}}u_0(y)\, dy=o(1)M  Z(x,t),
$$
if $\delta(t)t^{\frac\a{2s}}\to\infty$ as $t\to\infty$.
	
The choice $\delta(t)=t^{-\gamma}$ with $0<\gamma<\frac\a{2s}\big(1-\frac N\beta\big)$ fulfills all the required conditions. 
\end{proof}

We next turn our attention to initial data that do not decay necessarily as fast as the fundamental solution, but have still an integrable decay.

\begin{teo}
\label{thm:slow.decay} Let $u_0\in L^1(\R^N)\cap \mathcal{D}_\beta$, with $\beta\in(N,N+2s)$. Given  $\nu>0$ and  a function $h$ such that $h(t)=o\big(t^{\frac\a{N+2s-\beta}})$ as $t\to\infty$, then
\begin{equation*}
\label{eq-beta-small}
		u(x,t)=MZ(x,t)\big(1+o(1)\big)\quad\mbox{uniformly in }\nu t^{\frac\a{2s}}\le |x|\le h(t)\quad\text{as }t\to\infty.
\end{equation*}
\end{teo}
\begin{proof}We split the error as in the proof of Theorem~\ref{thm:fast.decay}. The terms ${\rm I}(x,t)$ and ${\rm III}(x,t)$ can be controlled exactly in the same way as there, just requiring  
\begin{equation}
\label{eq:first.conditions.delta}
\delta(t)=o(1),\quad\delta(t)t^{\frac\a{2s}}\to\infty\quad\text{as } t\to\infty, 
\end{equation}
since the above estimates of those terms do not use the decay of the initial datum. The difference arises in the estimate of~${\rm II}(x,t)$

We start from~\eqref{eq:estimate.II}, where we have only used that $\beta$ is nonnegative. If $\nu t^{\frac\a{2s}}\le |x|\le h(t)$, $\nu>0$,  using~\eqref{eq:estimate.F.below} we get
$$
{\rm II}(x,t)\le C(\delta(t))^{-\beta}t^{-\a}h(t)^{N+2s-\beta}MZ(x,t).
$$
Since we want to have ${\rm II}(x,t)=o(1)MZ(x,t)$, and $\beta<N+2s$, $h$ cannot grow arbitrarily fast. In order to have the desired bound we need 
\begin{equation}
\label{eq:second.condition.delta}
 h(t)=o\big((\delta(t))^{\frac\beta{N+2s-\beta}}t^{\frac\a{N+2s-\beta}}\big) \quad \text{as }t\to\infty.
\end{equation}

Our aim is to choose a function $\delta$ satisfying~\eqref{eq:first.conditions.delta}--\eqref{eq:second.condition.delta} allowing the fastest possible growth for $h$. Since $\delta(t)=o(1)$, it is clear from~\eqref{eq:second.condition.delta} that we cannot get from this proof anything better than $h(t)=o\big(t^{\frac\a{N+2s-\beta}})$. On the other hand, any such $h$ can be reached by means of the choice  
$$
(\delta(t))^{\frac{2\beta}{N+2s-\beta}}=h(t)t^{-\frac\a{N+2s-\beta}}.
$$ 
Indeed, this trivially guarantees $\delta=o(1)$ and the condition\eqref{eq:second.condition.delta}. As for the second condition in~\eqref{eq:first.conditions.delta}, since $h(t)\ge \nu t^{\frac\alpha{2s}}$, we have
$$
(\delta(t)t^{\frac\a{2s}})^{\frac{2\beta}{N+2s-\beta}}=h(t)t^{-\frac\a{N+2s-\beta}}t^{\frac{\a2\beta}{2s(N+2s-\beta)}}\ge\nu t^{\frac{\a(N+\beta)}{2s(N+2s-\beta)}}\to\infty,
$$
and we are done. 
\end{proof}

We end this section with a result that shows that the upper restriction in Theorem~\ref{thm:slow.decay} is not technical.

\begin{teo} Let $u_0\in L^1(\R^N)\cap \mathcal{D}_\beta$, with $\beta\in(N,N+2s)$, be such that $|x|^\beta u_0(x)\to A>0$ as $|x|\to\infty$. Given any $\nu>0$, and $h$  such that $h(t)\ge\nu t^{\frac\alpha{2s}}$ for all $t>0$,  $h(t)/t^{\frac\a{N+2s-\beta}}\to\infty$ as $t\to\infty$, then
	\begin{equation*}\label{eq-precise-far-far}
	u(x,t)=\frac A{|x|^\beta}\big(1+o(1)\big)\quad\mbox{uniformly in }|x|\ge h(t) \quad\text{as }t\to\infty.
	\end{equation*}

\end{teo}
\begin{proof} After adding and subtracting $A\int_{|y|>\delta|x|}Z(x-y,t)\frac{|x|^\beta}{|y|^\beta}\,dy$ for some fixed $\delta\in(0,1)$ to be chosen later, we have
\[
\begin{aligned}
\big||x|^\beta u(x,t)-A|&=\Big|\int Z(x-y,t)\big(|x|^\beta u_0(y)-A\big)\,dy\Big|\le {\rm I}(x,t)+{\rm II}(x,t)+{\rm III}(x,t)+{\rm IV}(x,t),\\
{\rm I}(x,t)&=
\int_{|y|<\delta |x|}Z(x-y,t)\big||x|^\beta u_0(y)-A\big|\,dy,\\
{\rm II}(x,t)&=A\int_{\scriptsize\begin{array}{c}|y|>\delta|x|\\
	|x-y|<t^{\frac\a{2s}}\log t\end{array}}Z(x-y,t)\Big|\frac{|x|^\beta}{|y|^\beta}-1\Big|\,dy,
\\
{\rm III}(x,t)&=A\int_{\scriptsize\begin{array}{c}|y|>\delta|x|\\
	|x-y|>t^{\frac\a{2s}}\log t\end{array}}Z(x-y,t)\Big|\frac{|x|^\beta}{|y|^\beta}-1\Big|\,dy,\\
{\rm IV}(x,t)&=\int_{|y|>\delta |x|}Z(x-y,t)\big||y|^\beta u_0(y)-A\big|\frac{|x|^\beta}{|y|^\beta} \,dy,
\end{aligned}
\]

If $|y|<\delta|x|$, $\delta\in (0,1)$, then $|x-y|>(1-\delta)|x|$. Therefore, using the estimate~\eqref{eq:bounds.F.infty},
$$
{\rm I}(x,t)\le  \frac{Ct^\a}{|x|^{N+2s}}
\int_{|y|<\delta |x|}\big||x|^\beta u_0(y)-A\big|\,dy.
$$
Since  $|x|>1$ in the region under consideration for $t$ large enough,  the integral on the right-hand side can be easily bounded by $C|x|^\beta$. Hence, if $|x|\ge h(t)$,
$$
{\rm I}(x,t)\le  \frac{Ct^\a}{|x|^{N+2s-\beta}}\le \frac{Ct^{\a}}{(h(t))^{N+2s-\beta}}\to0
$$
by hypothesis.

Under our assumptions $h(t)$ grows faster than $t^{\frac\a{2s}}\log t$. Hence, if $|x|\ge h(t)$ and $|x-y|<t^{\frac\a{2s}}\log t$, for all $t$ large $|x|>|x-y|$, and we have
\[\begin{aligned}
1-\frac{\beta t^{\frac\a{2s}}\log t}{h(t)}&\le1-\frac{\beta|x-y|}{|x|}\le \frac1{\big(1+\frac{|x-y|}{|x|}\big)^\beta}=\frac{|x|^\beta}{(|x|+|y-x|)^\beta}\\
&\le \frac{|x|^\beta}{|y|^\beta}
\le \frac{|x|^\beta}{(|x|-|y-x|)^\beta}=\frac1{\big(1-\frac{|x-y|}{|x|}\big)^\beta}\le1-\frac{\beta|x-y|}{|x|}\le 1+\frac{\beta t^{\frac\a{2s}}\log t}{h(t)}.
\end{aligned}
\]
Hence, $\displaystyle\Big|\frac{|x|^\beta}{|y|^\beta}-1\Big|\le \frac{\beta t^{\frac\a{2s}}\log t}{h(t)}$, so that	
\[
{\rm II}(x,t)\le \frac{Ct^{\frac\a{2s}}\log t}{h(t)}\int F(\xi)\,d\xi\to0\quad\mbox{as } t\to\infty,
\]
by the hypotheses on $h(t)$.
 
On the other hand, if $|y|>\delta|x|$, then $\Big|\frac{|x|^\beta}{|y|^\beta}-1\Big|\le \delta^{-\beta}+1$. Therefore,  after a change of variables,
\[
{\rm III}(x,t)\le C(1+\delta^{-\beta})\int_{|\xi|>\log t}F(\xi)\,d\xi\to0\quad\mbox{as } t\to\infty.
\]

Finally, if $|y|>\delta|x|\ge \delta h(t)$, on the one hand $\frac{|x|^\beta}{|y|^\beta}<\delta^{-\beta}$, and on the other hand $|y|\to\infty$ as $t\to\infty$. Therefore since $\big||y|^\beta u_0(y)-A\big|\to0$ as $|y|\to\infty$,  and $\int Z(\cdot,t)=1$, we conclude that ${\rm IV}(x,t)\to0$ as $t\to\infty$.
\end{proof}


\end{document}